\documentclass[12pt,a4paper]{amsart}
 \usepackage[T1]{fontenc}
\usepackage[top=35mm, bottom=35mm, left=30mm, right=30mm]{geometry}
\usepackage{amssymb}
\usepackage{mathtools}
\usepackage{verbatim,enumerate,enumitem}
\usepackage{xcolor}
\usepackage{amsthm}

\usepackage[looser]{newpxtext}
\usepackage[smallerops]{newpxmath}
\usepackage{fancyhdr}
\makeatletter
\def\@@and{\MakeLowercase{and}}
\makeatother

\usepackage[hidelinks,pagebackref]{hyperref}

\usepackage[nobysame,msc-links]{amsrefs} 

\theoremstyle{definition}
\newtheorem{defn}{Definition}[section]
\newtheorem{exam}[defn]{Example}

\newtheorem{rem}[defn]{Remark}

\theoremstyle{plain}
\newtheorem{thm}[defn]{Theorem}
\newtheorem{lem}[defn]{Lemma}
\newtheorem{prop}[defn]{Proposition}

\newcommand{\dd}{\mathop{}\!\mathrm{d}}

\DeclareMathOperator{\udens}{\overline{dens}}
\DeclareMathOperator{\ldens}{\underline{dens}}
\DeclareMathOperator{\uDens}{\overline{Dens}}
\DeclareMathOperator{\lDens}{\underline{Dens}}
\DeclareMathOperator{\Dens}{Dens}
\DeclareMathOperator{\dens}{dens}

\title[D\MakeLowercase{ensity properties of orbits for a hypercyclic operator}] 
{D\MakeLowercase{ensity properties of orbits for a hypercyclic operator on a }B\MakeLowercase{anach space}}

\author[J. L\MakeLowercase{i}]{J\MakeLowercase{ian} Li}
\address[J. Li]{Institute for  Mathematical Sciences and Artificial Intelligence \& Department of Mathematics,
	Shantou University, Shantou, 515821, Guangdong, China}
\email{lijian09@mail.ustc.edu.cn}
\urladdr{https://orcid.org/0000-0002-8724-3050}

\author[X. W\MakeLowercase{ang}]{X\MakeLowercase{insheng} Wang}
\address[X. Wang]{Department of Mathematics,
	Shantou University, Shantou, 515821, Guangdong, China}
\email{wangxs@stu.edu.cn}
\urladdr{https://orcid.org/0000-0002-5287-8902}

\author[J. Z\MakeLowercase{hao}]{J\MakeLowercase{ianjie} Zhao}
\address[J. Zhao]{School of Mathematics, Hangzhou Normal University, Hangzhou, 311121, Zhejiang, China}
\email{zjianjie@hznu.edu.cn}
\urladdr{https://orcid.org/0000-0003-0790-7038}


	

\date{\today}

\begin{document}

\begin{abstract}
We study density properties of orbits for a hypercyclic operator $T$ on a separable Banach space $X$, and show that exactly one of the following four cases holds: 
(1) every vector in $X$ is asymptotic to zero with density one;
(2) generic vectors in $X$ are distributionally irregular of type $1$;
(3) generic vectors in $X$ are distributionally irregular of type $2\frac{1}{2}$ and no hypercyclic vector is distributionally irregular of type $1$;
(4) every hypercyclic vector in $X$ is divergent to infinity with density one. 
We also present some examples concerned with weighted backward shifts on $\ell^p$ to show that all the above four cases can occur. Furthermore, we show that similar results hold for $C_0$-semigroups.
\end{abstract}

\maketitle

\section{Introduction}
Dynamical systems is a branch of mathematics concerned with the asymptotic behavior of systems. 
The study of orbits generated by the system plays an important role for exploring the complexity of the asymptotic behavior. 
Plenty of results related to properties of orbits have been obtained for various systems under different settings. 
For the dynamical system generated by a linear operator on some linear space with certain structure, it seems that there is no complexity due to the linearity.
However, even for such a system, there can be abundant complex phenomena, meaning that it is significant and interesting to consider linear dynamics.  

One of the central concepts in the study of linear dynamics is the so-called hypercyclicity, which is closely related to numerous important properties of linear dynamics. 
Let $X$ be a Banach space and $T\colon X\to X$ be a bounded linear operator. 
A vector $x\in X$ is said to be \emph{hypercyclic}
if the orbit of $x$, $\{T^nx\colon n\geq 0\}$, is dense in $X$, and the operator $T$ is said to be \emph{hypercyclic} if there exists some hypercyclic vector in $X$. 
The set of all hypercyclic vectors of $T$ is denoted by $HC(T)$. Note that the concept of hypercyclicity is in fact the so-called transitivity for general topological dynamical systems. 
It is important to point out that linear structure can be made full use of for linear dynamics, compared to general dynamical systems, which means that we can consider the properties of linear dynamics from some new points of view. 
For a complete discussion on the dynamics of operators, the reader can refer to \cite{BM2009} and \cite{GP2011} and references therein.

As mentioned above, hypercyclic operators play a significant role in linear dynamics, making it useful and interesting to study their properties from various perspectives. 
In \cite{L2024} the first named author of this paper studied the behaviors of orbits in the mean sense of a hypercyclic operator on a Banach space.
In this paper, another aspect of the hypercyclic operator is considered, and we will study the density properties of orbits for such operators.

First of all, let us recall some concepts related to the topic of this paper. 
For a subset $A$ of $\mathbb{N}$, the \emph{upper density} and \emph{lower density} of $A$ are defined by
\[
\udens(A)=\limsup_{n\to\infty} \frac{ \#(A\cap [1,n])}{n},
\]
and 
\[
\ldens(A)=\liminf_{n\to\infty} \frac{ \#(A\cap [1,n])}{n},
\]
respectively, where $\#(\cdot)$  denotes the number of elements of a finite set.
If $\udens(A)=\ldens(A)$, then we say that the \emph{density} of $A$ exists, and denote it by $\dens(A)$.

Let $X$ be a Banach space with the norm $\Vert\cdot\Vert$ and $T\colon X\to X$ be a bounded linear operator from $X$ to itself. 
This paper focuses on the following properties of hypercyclic operators in terms of density. 
Following \cite{BBMP2011},  a vector $x\in X$ is said to be \emph{distributionally irregular of type $1$} for $T$ if there exist two subsets $A$ and $B$ of $\mathbb{N}$ such that $\udens(A)=\udens(B)=1$, $\lim\limits_{A\ni n\to\infty} \Vert T^n x\Vert =0$ and $\lim\limits_{B\ni n\to\infty} \Vert T^n x\Vert =\infty$. 
Inspired by the concept of distributional chaos of type $2\frac{1}{2}$ given in \cite{BBPW2018}, we introduce the concept of distributionally irregular of type $2\frac{1}{2}$. 
A vector $x\in X$ is said to be \emph{distributionally irregular of type $2\frac{1}{2}$} for $T$ if there exist two subsets $A$ and $B$ of $\mathbb{N}$ such that 
$\udens(A)\udens(B)>0$, $\lim\limits_{A\ni n\to\infty} \Vert T^n x\Vert =0$ and $\lim\limits_{B\ni n\to\infty} \Vert T^n x\Vert =\infty$. 
A vector $x\in X$ is said to be \emph{asymptotic to zero with density one} if there exists a subset $A$ of $\mathbb{N}$ with $\dens(A)=1$ such that $\lim\limits_{A\ni n\to\infty}\Vert T^n x\Vert =0$,
and \emph{divergent to infinity with density one}
if there exists a subset $B$ of $\mathbb{N}$ with $\dens(B)=1$ such that $\lim\limits_{B\ni n\to\infty}\Vert T^n x\Vert =\infty$.

In addition, we say that a property P on vectors in $X$ is \emph{generic} if the collection of vectors which satisfy the property P contains a dense $G_\delta$ subset of $X$.  
The following theorem is our main result in this paper.

\begin{thm}\label{thm:main-result}
If $T$ is a hypercyclic operator on a separable Banach space $X$, then exactly one of the following four assertions holds: 
\begin{enumerate}
    \item every vector in $X$ is asymptotic to zero with density one;
    \item generic vectors in $X$ are distributionally irregular of type $1$;
    \item generic vectors in $X$ are distributionally irregular of type $2\frac{1}{2}$, and no hypercyclic vector is distributionally irregular of type $1$;
    \item every hypercyclic vector in $X$ is divergent to infinity with density one.
\end{enumerate}
\end{thm}

The paper is organized as follows. Theorem~\ref{thm:main-result} will be proved in Section 2.
In Section 3, we  present some examples concerned with weighted backward shifts on $\ell^p$ to show that all four cases in Theorem~\ref{thm:main-result} can occur. 
In Section 4, similar results for $C_0$-semigroups will be discussed.

\section{Proof of Theorem~\ref{thm:main-result}}

The aim of this section is to prove Theorem~\ref{thm:main-result}.
We begin with the following standard result on density; therefore, its proof is omitted.  
\begin{lem}\label{lem:density}
Let $(a_n)_n$ be a sequence in $\mathbb{R}$. Then
\begin{enumerate}\label{Analysis: goes to 0 in density one}
    \item there exists a subset $A$ of $\mathbb{N}$ with 
    $\udens(A)=1$ such that $\lim_{A\ni n\to\infty}|a_n|=0$
    if and only if for every $\varepsilon>0$,
    $\udens( \{n\in \mathbb{N}: |a_n|<\varepsilon \} )=1$;
    \item\label{lem_item2} there exists a subset $B$ of $\mathbb{N}$ with 
    $\dens(B)=1$ such that $\lim_{B\ni n\to\infty}|a_n|=0$
    if and only if for every $\varepsilon>0$,
    $\dens( \{n\in \mathbb{N}: |a_n|<\varepsilon \} )=1$. 
\end{enumerate}
\end{lem}

By a \emph{topological dynamical system}, we mean a pair $(X,T)$,
where $X$ is a metric space with the metric $d$ and $T\colon X\to X$ is a continuous map.
Following~\cite{F1951} and~\cite{LTY2015}, 
we say that a topological dynamical system $(X,T)$ 
is \emph{mean-L-stable} 
if for every $\varepsilon>0$ there exists a $\delta>0$
such that for any $x,y\in X$ with $d(x,y)<\delta$, one has
\[
    \udens(\{n\in\mathbb{N}\colon d(T^nx,T^ny)\geq \varepsilon\})<\varepsilon,
\]
and \emph{mean-L-unstable}
if  there exists a $\delta>0$
such that for any non-empty open subset $U$ of $X$ there exist $x,y\in U$ such that
\[
    \udens(\{n\in\mathbb{N}\colon d(T^nx,T^ny)> \delta\})\geq \delta.
\]

We now state the following dichotomy theorem.
\begin{thm}[\cite{JL2022}*{Theorem 5.13}]
\label{thm:meanLstableunstable}
Let $T$ be a bounded linear operator on a Banach space $X$. Then $T$ is either mean-L-stable or mean-L-unstable.
\end{thm}

We have an equivalent characterization concerned with mean-L-stability in terms of density under some conditions, the following proposition plays a key role in such a characterization.
\begin{prop} \label{prop:mean-l-stable}
Let $T$ be a bounded linear operator on a Banach space $X$.
Assume that $T$ is mean-L-stable. 
Then we have the following:
\begin{enumerate}
    \item For a vector $x\in X$, if $\liminf_{n\to\infty}\Vert T^nx\Vert= 0$,
    then $x$ is asymptotic to zero with density one.
    \item The collection of vectors in $X$ which are asymptotic to zero with density one is a closed subspace of $X$.
\end{enumerate}
\end{prop}
\begin{proof}
\begin{enumerate}[fullwidth]
    \item Let $x\in X$ be a vector satisfying $\liminf_{n\to\infty}\Vert T^nx\Vert= 0$. Fix any $\varepsilon>0$. 
    Note that $T$ is mean-L-stable, which means for any $0<\varepsilon'<\varepsilon$ there exists $\delta>0$ such that for any $x$, $y\in X$ with $\Vert x-y\Vert <\delta$, we have
    \[
        \udens(\{n\in\mathbb{N}\colon \Vert T^nx- T^ny\Vert \geq \varepsilon'\})<\varepsilon'.
    \]
    By the assumption on $x$, there exists $k\in\mathbb{N}$ such that $\Vert T^kx\Vert <\delta$. Then for $T^kx$ and $\mathbf{0}$, we know that
    \[
        \udens(\{n\in\mathbb{N}\colon\Vert T^{n+k}x\Vert \geq \varepsilon'\})<\varepsilon',
    \]
    implying that
    \[
        \ldens(\{n\in\mathbb{N}\colon \Vert T^{n+k}x\Vert < \varepsilon'\})\geq1-\varepsilon'.
    \]
    Noticing that $\varepsilon'<\varepsilon$ we have
    \begin{align*}
        \ldens(\{n\in\mathbb{N}\colon \Vert T^{n}x\Vert <\varepsilon\})
        &\geq \ldens(\{n\in\mathbb{N}\colon \Vert T^{n}x\Vert <\varepsilon'\})\\
        &\geq\ldens(\{n\in\mathbb{N}\colon \Vert T^{n+k}x\Vert < \varepsilon'\})
        >1-\varepsilon'.
    \end{align*}
    Letting $\varepsilon'\to 0$ we obtain
    \[
        \ldens(\{n\in\mathbb{N}\colon \Vert T^{n}x\Vert <\varepsilon\})\geq 1,
    \]
    meaning that
    \[
        \dens(\{n\in\mathbb{N}\colon \Vert T^{n}x\Vert <\varepsilon\})=1.
    \]
    Noticing the arbitrariness of $\varepsilon$ and using Lemma \ref{lem:density} \eqref{lem_item2} we know that there exists a subset $A$ of $\mathbb{N}$ with $\dens(A)=1$ such that $\lim_{A\ni n\to\infty} \Vert T^n x\Vert =0$.
    Then $x$ is asymptotic to zero with density one.
    
    \item Let $X_0$ be the collection of vectors in $X$ which are asymptotic to zero with density one.
    For any $x, y\in X_0$, there exist two subsets $A_x, A_y\subset \mathbb{N}$ with $\dens(A_x)=1=\dens(A_y)$ such that 
    \[
      \lim_{A_x\ni n\to\infty} \Vert T^n x\Vert =0 \ \text{and}  \ \lim_{A_y\ni n\to\infty} \Vert T^n y\Vert =0. 
    \]
    Let $A=A_x\cap A_y$. Since $\dens(A_x)=\dens(A_y)=1$, it is easy to see that $\dens(A)=1$.
    On the other hand, clearly we have
    \begin{align*}
      0\leq&\lim_{A\ni n\to\infty} \Vert T^n (x+y)\Vert
      \leq\lim_{A\ni n\to\infty}\left(\Vert T^n x\Vert+\Vert T^n y\Vert\right)\\
      =&\lim_{A\ni n\to\infty}\Vert T^n x\Vert+\lim_{A\ni n\to\infty}\Vert T^n y\Vert=0,
    \end{align*}
    which means that $x+y\in X_0$. It is easy to check that $rx\in X_0$ for any $r\in\mathbb{K}$, where $\mathbb{K}=\mathbb{R}\text{ or }\mathbb{C}$. 
    Now we have proved that $X_0$ is a subspace. In the following we will prove that $X_0$ is closed. Let $\{x_k\}_{k\in\mathbb{N}}$ be a sequence in $X_0$ with $\lim_{k\to\infty}x_k=x\in X$. Fix any $\varepsilon>0$. For any $0<\varepsilon'<\frac{\varepsilon}{2}$, noticing that $T$ is mean-L-stable, there exists $\delta>0$ such that for any $y$, $z\in X$ with $\Vert y-z\Vert<\delta$,
    \[
        \udens(\{n\in\mathbb{N}\colon\Vert T^n(y-z)\Vert\geq\varepsilon'\})<\varepsilon',
    \]
    that is
    \[
        \ldens(\{n\in\mathbb{N}\colon\Vert T^n(y-z)\Vert<\varepsilon'\})\geq 1-\varepsilon'.
    \]
    Choose $m>0$ such that $\Vert x_m-x\Vert<\delta$. Then we have 
    \begin{align*}
        &\ldens(\{n\in\mathbb{N}\colon\Vert T^n(x_m-x)\Vert<\varepsilon'\})\ge 1-\varepsilon'.
    \end{align*}
    Since $x_m\in X_0$, we know that $\dens(\{n\in\mathbb{N}\colon\Vert T^nx_m \Vert<\varepsilon'\})=1$. 
    Then we have
    \begin{align*}
    \ldens(\{n\in \mathbb{N}\colon \Vert T^nx \Vert <\varepsilon)\}
      &\ge  \ldens\bigl(\{n\in \mathbb{N}: \Vert T^nx_m \Vert <\varepsilon'\} \\
      &\qquad\qquad  \cap \{ n\in \mathbb{N}: \Vert T^nx_m-T^nx \Vert <\varepsilon'\} \bigr)\\ 
      &\ge 1-\varepsilon'.
    \end{align*}
    Letting $\varepsilon'\to 0$, we obtain that
    \[
     \ldens(\{n\in \mathbb{N}: \Vert T^nx \Vert <\varepsilon)\}=1. 
    \]
    This shows that $x\in X_0$ and then $X_0$ is closed. \qedhere
\end{enumerate}
\end{proof}

Now we can give the following characterization of mean-L-stability for hypercyclic operators.

\begin{prop}\label{prop:mean-L-stable} 
Let $T\colon X\to X$ be a hypercyclic operator on a separable Banach space $X$.  
Then $T$ is mean-L-stable if and only if every vector in $X$ is asymptotic to zero with  density one. 
\end{prop}
\begin{proof}
 ``$\Rightarrow$"
For every hypercyclic vector $x\in X$,
 as $\liminf_{n\to\infty} \Vert T^nx\Vert =0$,
 by Proposition~\ref{prop:mean-l-stable} (1),
 $x$ is asymptotic to zero with density one.
By the fact that $HC(T)$ is dense in $X$
and Proposition~\ref{prop:mean-l-stable} (2), 
every vector in $X$ is asymptotic to zero with density one.

``$\Leftarrow$'' By the assumption and Lemma \ref{lem:density} \eqref{lem_item2}, for any $z\in X$ and any $\varepsilon>0$ we have
\[
    \dens(\{n\in\mathbb{N}\colon\Vert T^nz \Vert<\varepsilon\})=1.
\]
Then for any $x$, $y\in X$ and any $\varepsilon>0$ we have
\begin{align*}
\dens(\{n\in\mathbb{N}\colon\Vert T^n(x-y)\Vert<\varepsilon\})
=1.
\end{align*}
The arbitrariness of $\varepsilon$ shows that $T$ is mean-L-stable.
\end{proof}

A characterization of mean-L-unstability can also be formulated, involving the so-called distributional sensitivity. 
We say that a topological dynamical system 
$(X,T)$ is \emph{distributionally sensitive} if there exists $\delta>0$ such that
for any $x\in X$ and $\varepsilon>0$, there exists $y\in X$ with $d(x,y)<\varepsilon$ such that
\[
    \udens(\{n\in\mathbb{N}\colon d(T^nx,T^ny)> \delta\})=1.
\]
In addition, for a bounded linear operator $T$ on a Banach space $X$, we say that a vector $x\in X$ is \emph{distributionally unbounded} if there exists a subset $B$ of $\mathbb{N}$ with $\udens(B)=1$ and $\lim_{B\ni n\to\infty} \Vert T^nx \Vert =\infty$.
The following result follows from \cite{JL2022}*{Proposition 5.33}. 

\begin{prop}\label{prop:D-densitive}
Let $T$ be a bounded linear operator on a Banach space $X$.
If  $T$ is distributionally sensitive, then generic vectors in $X$ are distributionally unbounded.
\end{prop}

The following result is essentially contained in \cite{BBMP2013}*{Proposition 8}, see \cite{JL2022}*{Theorem 5.41} for the version here, which establishes a characterization of mean-L-unstability.

\begin{thm}\label{thm:mean-L-unstable}
Let $T$ be a bounded linear operator on a Banach space $X$.
Then $T$ is mean-L-unstable if and only if it is distributionally sensitive.
\end{thm}

Now we can give the following dichotomy for hypercyclic operators.
\begin{prop}
Let $T$ be a hypercyclic operator on a separable Banach space $X$.
Then either every vector in $X$ is asymptotic to zero with density one, or generic vectors in $X$ are distributionally unbounded.
\end{prop}
\begin{proof}
    By Theorem~\ref{thm:meanLstableunstable} we know that $T$ is either mean-L-stable or mean-L-unstable. 
    On the one hand, if $T$ is mean-L-stable, note that $T$ is hypercyclic, by Proposition~\ref{prop:mean-L-stable} we know that every vector in $X$ is asymptotic to zero with density one.
    On the other hand, if $T$ is mean-L-unstable, by Theorem~\ref{thm:mean-L-unstable} we know that $T$ is distributionally sensitive. Then by Proposition \ref{prop:D-densitive} we obtain that generic vectors in $X$ are distributionally unbounded.
\end{proof}

Let $T$ be a hypercyclic operator on a Banach space $X$. 
For the proof of our main result (Theorem \ref{thm:main-result}), we also need a parameter $c(T) \in [0, 1]$ associated with $T$, which was firstly introduced by Grivaux and Matheron in \cite{GM2014}.
For each $R>0$, define 
\[
c_R(T) =\sup_{x\in HC(T)} \udens(\{n\in\mathbb{N}\colon
\Vert T^nx\Vert<R\}),
\]
and  
\[
c(T)=\sup_{R>0} c_R(T).
\]

\begin{rem}\label{rem:c-T-generic}
It is shown in \cite{GM2014}*{Remark 4.6}
that there exists a dense $G_\delta$ subset $X_0$ of $HC(T)$ such that for any $x\in X_0$ and $R>0$, 
$\udens(\{n\in\mathbb{N}\colon
\Vert T^nx\Vert<R\})=c(T)$.
Since $HC(T)$ is a  dense $G_\delta$ subset of $X$,
$X_0$ is also a dense $G_\delta$ subset of $X$.
\end{rem}  

Concerned with $c(T)$, we have the following proposition, which is an important tool in the proof of Theorem \ref{thm:main-result}.
\begin{prop}\label{prop:c-T-eq}
Let $T$ be a hypercyclic operator on a separable Banach space $X$.
Then 
\[
c(T)=\sup\{ \lambda\in[0,1]\colon 
\exists x\in HC(T)\, \&\, R>0 \text{ s.t. }
\udens(\{n\in\mathbb{N}\colon
\Vert T^nx\Vert<R\})\geq  \lambda\}
\]
and for every hypercyclic vector $y\in HC(T)$ and $R>0$, 
\[
  \ldens(\{n\in\mathbb{N}\colon
\Vert T^ny\Vert\geq R\})\geq 1-c(T).  
\]
\end{prop}
\begin{proof}
Let 
\[
    A= \{ \lambda\in[0,1]\colon  \exists x\in HC(T) \, \&\, R>0 \text{ s.t. } \udens(\{n\in\mathbb{N}\colon \Vert T^nx\Vert<R\})\geq  \lambda\}.
\]
For any $\lambda\in A$, 
there exist $x\in HC(T)$ and $R>0$ such that
$\udens(\{n\in\mathbb{N}\colon \Vert T^nx\Vert<R\})\geq  \lambda$.
Let $\lambda_0=\udens(\{n\in\mathbb{N}\colon \Vert T^nx\Vert<R\})$.
It is clear that $\lambda\leq \lambda_0$ and $\lambda_0\in A$.
By the definition of $c_R(T)$ and $c(T)$, one has $\lambda_0\leq c_R(T)\leq c(T)$. 
Then $c(T)$ is a upper bound for $A$.
For every $\varepsilon>0$, there exists $R_0>0$ such that 
$c_{R_0}(T)>c(T)-\varepsilon$ and then
there exists $x\in HC(T)$ such that $\udens(\{n\in\mathbb{N}\colon
\Vert T^nx\Vert<R_0\})\geq c(T)-\varepsilon$.
So $\max\{ c(T)-\varepsilon, 0\}\in A$.
This implies that $c(T)=\sup A$.

Fix $y\in HC(T)$ and $R>0$. One has 
$\udens(\{n\in\mathbb{N}\colon
\Vert T^ny\Vert<R\})\in A$ and then
\[ 
\udens(\{n\in\mathbb{N}\colon \Vert T^ny\Vert<R\})\leq c(T).
\]
Therefore,
\begin{align*}
     \ldens(\{n\in\mathbb{N}\colon \Vert T^ny\Vert\geq R\})&\geq  1- \udens(\{n\in\mathbb{N}\colon \Vert T^ny\Vert<R\})\\
     &\geq 1-c(T). \qedhere 
\end{align*}

\end{proof} 
 
Now we are ready to prove Theorem~\ref{thm:main-result}.

\begin{proof}[Proof of Theorem~\ref{thm:main-result}]
If $T$ is mean-L-stable, then by Proposition~\ref{prop:mean-L-stable} every vector in $X$ is asymptotic to zero with  density one.
So, in this case, assertion (1) holds.

If $T$ is mean-L-unstable and $c(T)=1$, then combining Proposition~\ref{prop:D-densitive} and Theorem~\ref{thm:mean-L-unstable} there exists a dense $G_\delta$ subset $X_1$ of $X$ such that every vector in $X_1$ is distributionally unbounded,
and by Remark~\ref{rem:c-T-generic}, there exists a dense $G_\delta$ subset $X_2$ of $X$ such that for every $x\in X_2$ and $R>0$, 
$\udens(\{n\in\mathbb{N}\colon \Vert T^nx\Vert<R\})=1$.
Let $X_0=X_1\cap X_2$. Then $X_0$ is a dense $G_\delta$ subset of $X$ and every vector in $X_0$ is distributionally irregular of type $1$. So, in this case, assertion (2) holds.

If $T$ is mean-L-unstable and $0<c(T)<1$, then there exists a dense $G_\delta$ subset $X_1$ of $X$ such that every vector in $X_1$ is distributionally unbounded,
and by Remark~\ref{rem:c-T-generic}, there exists a dense $G_\delta$ subset $X_2$ of $X$ such that for every $x\in X_2$ and $R>0$, 
$\udens(\{n\in\mathbb{N}\colon \Vert T^nx\Vert<R\})=c(T)>0$.
Let $X_0=X_1\cap X_2$. 
Then $X_0$ is a dense $G_\delta$ subset of $X$ and every vector in $X_0$ is distributionally irregular of type $2\frac{1}{2}$. 
By Proposition~\ref{prop:c-T-eq}, for every hypercyclic vector $y\in HC(T)$ and $R>0$, 
\[
  \ldens(\{n\in\mathbb{N}\colon \Vert T^ny\Vert\geq R\})\geq 1-c(T)>0.  
\]
So every hypercyclic vector in $X$ is not distributionally irregular of type $1$. 
So, in this case, assertion (3) holds.

If $T$ is mean-L-unstable and $c(T)=0$,
then by Proposition~\ref{prop:c-T-eq}, for every hypercyclic vector $y\in HC(T)$ and $R>0$, 
\[
  \ldens(\{n\in\mathbb{N}\colon \Vert T^ny\Vert\geq R\})\geq 1-c(T)=1.  
\]
So every hypercyclic vector in $X$ is divergent to infinity with density one. So, in this case, assertion (4) holds.
\end{proof}

\begin{rem}
By the proof of Theorem~\ref{thm:main-result}, we can replace the case (3) in Theorem~\ref{thm:main-result} by
\begin{enumerate}
    \item[(3$^\prime$)] there exists a dense $G_\delta$ subset $X_0$ of $ X$ and a constant $c\in(0,1)$ such that for every $x\in X_0$, there exist two subsets $A$ and $B$ of $\mathbb{N}$ with $\udens(A)=c$ and $\udens(B)=1$ such that $\lim\limits_{A\ni n\to\infty} \Vert T^n x\Vert =0$ and $\lim\limits_{B\ni n\to\infty} \Vert T^n x\Vert =\infty$, and for any hypercyclic vector $y\in HC(T)$  there exists  a subset $C$ of $\mathbb{N}$ with $\ldens (C)=1-c$ such that 
 $\lim\limits_{C\ni n\to\infty} \Vert T^n y\Vert =\infty$.
\end{enumerate}
\end{rem}

\section{Some examples}
In this section, we will present some examples concerned with weighted backward shifts on $\ell^p$ to show that all four cases in Theorem~\ref{thm:main-result} can occur. 

For a positive weight sequence $v=(v_j)_{j\geq 1}$ and $1\leq p<\infty$,
let
\[
    \ell^p(v,\mathbb{N})=\biggl\{(x_j)_{j\geq 1} \colon  \sum_{j=1}^\infty |x_j|^p v_j<\infty\biggr\}
\]
be the weighted $\ell^p$-space.
The norm on $\ell^p(v,\mathbb{N})$ is defined as follows:
\[
    \Vert x\Vert = \biggl(\sum_{j=1}^\infty |x_j|^p v_j\biggr)^{\frac{1}{p}},\quad \forall x=(x_j)_{j\geq 1}\in \ell^p(v,\mathbb{N}).
\]
The \emph{unilateral backward shift} $B$ on $\ell^p(v,\mathbb{N})$ is defined as 
\[
   (Bx)_j=x_{j+1},\quad \forall x=(x_j)_{j\geq 1}\in \ell^p(v,\mathbb{N}),\ j\geq 1.  
\]
It is not hard to show that the map $B\colon \ell^p(v,\mathbb{N})\to \ell^p(v,\mathbb{N})$ is well-defined (and, equivalently, continuous) if and only if $\sup_{j\geq 1}\frac{v_j}{v_{j+1}}<\infty$.

\begin{exam}\label{exam:case-1}
By \cite{MOP2013}*{Theorem 2.1},
there exists a positive weight sequence $v$ such that the operator 
$B\colon \ell^p(v,\mathbb{N})\to \ell^p(v,\mathbb{N})$ is mixing, $1\leq p<\infty$, and satisfies the case (1) of Theorem~\ref{thm:main-result}.
\end{exam}

\begin{rem}
By \cite{BBMP2011}*{Theorem 35},
in every infinite dimensional separable Banach space there exists a hypercyclic and distributionally chaotic operator which admits a dense distributionally irregular manifold,
which satisfies the case (2) of Theorem~\ref{thm:main-result}.
\end{rem}

Now we consider the bilateral case. 
For a positive weight sequence $v=(v_j)_{j\in\mathbb{Z}}$ and $1\leq p<\infty$,
let
\[
    \ell^p(v,\mathbb{Z})=\biggl\{(x_j)_{j\in\mathbb{Z}} \colon  \sum_{j\in \mathbb{Z}} |x_j|^p v_j<\infty\biggr\},
\]
be the weighted $\ell^p$-space.
The norm on $\ell^p(v,\mathbb{Z})$ is defined as follows:
\[
    \Vert x\Vert = \biggl(\sum_{j\in\mathbb{Z}} |x_j|^p v_j\biggr)^{\frac{1}{p}},\quad \forall x=(x_j)_{j\in\mathbb{Z}}\in \ell^p(v,\mathbb{Z}).
\]
The \emph{bilateral backward shift} $B$ on $\ell^p(v,\mathbb{Z})$ is defined as 
\[
   (Bx)_j=x_{j+1},\quad \forall x=(x_j)_{j\in\mathbb{Z}}\in \ell^p(v,\mathbb{Z}),\ j\in\mathbb{Z}.  
\]
Similarly, the map $B\colon \ell^p(v,\mathbb{Z})\to \ell^p(v,\mathbb{Z})$ is well-defined (and, equivalently, continuous) if and only if $\sup_{j\in\mathbb{Z}}\frac{v_j}{v_{j+1}}<\infty$.

\begin{rem}
By \cite{MOP2013}*{Example 3.5},
there exists a positive weight sequence $v$ such that the operator 
$B\colon \ell^p(v,\mathbb{Z})\to \ell^p(v,\mathbb{Z})$ is hypercyclic, $1\leq p<\infty$, and every non-zero vector in $X$ is distributionally irregular of type $1$. Note that this is a special situation satisfying case (2) in Theorem~\ref{thm:main-result}.
\end{rem}

\begin{rem}
Bayart and Ruzsa in \cite{BR2015}*{Section 6} constructed 
a frequently hypercyclic weighted shift on $c_0(\mathbb{Z})$ with $0<c(T)<1$, see also the proof of Theorem 1.7 in \cite{GM2014}. 
This meets the case (3) in Theorem~\ref{thm:main-result}.
\end{rem}

In the following example, we will construct a hypercyclic operator $T$ on the  weighted $\ell^p$-space
such that $c(T)\in (0,1)$ and for every non-zero vector $x\in X$ there exists $\tau_x>0$ such that $\ldens(\{j\geq 0\colon \Vert T^j x\Vert \geq \tau_x\})\geq 3/7$.
\begin{exam}\label{exam:case3}
Let $a_k=8^k$ for each $k\in\mathbb{N}$.
Define a positive weight sequence $v=(v_j)_{j\in\mathbb{Z}}$ as follows:
\[
v_j= \begin{cases}
\frac{1}{j}, & j\in\mathbb{N},\\
1,            & 0\leq -j\leq a_1,\\ 
\frac{1}{2^{a_k+j}}   & a_k< -j \leq a_k+k,\\
\frac{1}{2^k}, & a_k+k < -j \leq 2 a_k-k,\\
\frac{1}{2^{2a_k+j}}   & 2a_k-k< -j \leq 2a_k,\\
1 & 2a_k<-j\leq a_{k+1}.
\end{cases}
\]
Consider the  weighted $\ell^p$-space $\ell^p(v,\mathbb{Z})$ and the bilateral backward shift $B$ on $\ell^p(v,\mathbb{Z})$.
As $\sup_{j\in\mathbb{Z}}\frac{v_j}{v_{j+1}}\leq 2$, $B$ is well-defined.
By~\cite{GP2011}*{Theorem 4.12}, the bilateral backward shift $B$ is hypercyclic.

Consider the vector $e_0$ in $\ell^p(v,\mathbb{Z})$,
i.e., $e_0(0)=1$ and $e_0(i)=0$ for all $i\in\mathbb{Z}\setminus\{0\}$. 
For any $n\in\mathbb{N}$, $\Vert B^n e_0\Vert =v_{-n}^{1/p}$.
Note that 
\[
\{n\in\mathbb{N}\colon v_{-n}=1\} \supset \bigcup_{k=1}^\infty [2a_k,a_{k+1}] 
\]
and for every $m\in \mathbb{N}$
\[
\{n\in\mathbb{N}\colon v_{-n}\leq \tfrac{1}{2^m}\}  \supset \bigcup_{k=m}^\infty [a_k+k+1,2a_k-k].
\]
So 
\[
\ldens(\{n\in\mathbb{N}\colon \Vert B^n e_0\Vert\geq 1\})\geq \tfrac{3}{7},
\]
and for every $\varepsilon>0$
\[
    \udens(\{n\in\mathbb{N}\colon \Vert B^n e_0\Vert\leq \varepsilon\})\geq \tfrac{4}{7}.
\]
For every non-zero $x\in \ell^p(v,\mathbb{Z})$,
there exists some $r\in\mathbb{Z}$ such that $x_r\neq 0$. 
Since
\[
    \|B^nx\|^p\geq |x_r|^pv_{r-n}
\]
and the lower density is invariant under finite translations, we obtain
\[
    \ldens(\{n\in\mathbb{N}\colon \Vert B^n x\Vert\geq |x_r|\})\geq \tfrac{3}{7}.
\]
Note that the collection 
\[
    \{x\in X\colon \udens(\{n\in\mathbb{N}\colon \Vert B^n x\Vert\leq \varepsilon\})\geq \tfrac{4}{7}, \forall \varepsilon>0\}
\]
is a dense $G_\delta$ subset of $X$,
as it contains all the vectors with only finite non-zero coordinates.
\end{exam}

\begin{rem}
Menet in \cite{M2017} constructed a Devaney chaotic operator on $\ell^p$ such that for any non-zero vector $x$, there exists $\tau>0$ such that $\ldens(\{j\geq 0\colon \Vert T^j x\Vert \geq \tau\})=1$.
Then $c(T)=0$. 
This satisfies case (4) in Theorem~\ref{thm:main-result}.
\end{rem}

The hypercyclic operator constructed in the following example is adapted from that in \cite{L2024}, which has the property that every non-zero vector goes to infinity with density one. 

\begin{exam}\label{exam:case4}
Let $a_k=8^k$ for each $k\in\mathbb{N}$.
Define a weight sequence $v=(v_j)_{j\in\mathbb{Z}}$ as follows:
\[
v_j= \begin{cases}
\frac{1}{j}, & j\in\mathbb{N},\\
1,            & 0\leq -j\leq a_1,\\ 
2^{-k}   & a_k< -j \leq a_k+2k,\\
2^{-j-a_k-3k}  & a_k+2k< -j \leq a_k+4k,\\
2^{k} & a_k+4k <-j\leq a_{k+1}.
\end{cases}
\]
Consider the  weighted $\ell^p$-space $\ell^p(v,\mathbb{Z})$ and the bilateral backward shift $B$ on $\ell^p(v,\mathbb{Z})$.
As $\frac{v_j}{v_{j+1}}\leq 2$ for all $j\in \mathbb{Z}$, $B$ is well-defined.
By~\cite{GP2011}*{Theorem 4.12}, the bilateral backward shift  $B$ is hypercyclic.

Consider the vector $e_0$ in $\ell^p(v,\mathbb{Z})$,
i.e., $e_0(0)=1$ and $e_0(i)=0$ for all $i\in\mathbb{Z}\setminus\{0\}$. 
For any $n\in\mathbb{N}$, $\Vert B^n e_0\Vert =v_{-n}^{1/p}$.
Note that for every $m\in\mathbb{N}$,
\[
\{n\in\mathbb{N}\colon v_{-n}\geq 2^{m}\} \supset \bigcup_{k=m}^\infty [a_k+4k+1,a_{k+1}].
\]
So for every $m\in\mathbb{N}$,
\[
    \ldens(\{n\in\mathbb{N}\colon \Vert B^n e_0\Vert^p\geq  2^{m} \})=1.
\]
For every non-zero $x\in \ell^p(v,\mathbb{Z})$,
there exists some $r\in\mathbb{Z}$ such that $x_r\neq 0$,
Since 
\[
    \|B^nx\|^p\geq |x_r|^p v_{r-n}
\] 
and the lower density is invariant under finite translations, we obtain that 
for every $R>0$,  
\[
    \ldens(\{n\in\mathbb{N}\colon \Vert B^n x\Vert\geq R\})=1.
\]
\end{exam}

\section{Results for \texorpdfstring{$C_0$-semigroups}{C0-semigroup}}
The dynamics of $C_0$-semigroups were systematically studied for the first time in \cite{DSW1997}. For further details, see \cite{GP2011}*{Chapter 7}. In this section, we discuss results for $C_0$-semigroups that are analogous to those presented in the previous sections.

We say that a one-parameter family $(T_t)_{t\geq 0}$ of bounded linear operators on a Banach space $X$ 
is a \emph{$C_0$-semigroup} if $T_0=I$, $T_{t+s}=T_tT_s$ for all $t,s\geq 0$ and $\lim_{t\to s}T_t x=T_sx$ for all $x\in X$ and $s\geq 0$.
It is easy to see that such a $C_0$-semigroup is \emph{locally equicontinuous}, that is, for any $s>0$,
\[
\sup_{t\in [0,s]} \Vert  T_t\Vert<\infty.
\]

Let $(T_t)_{t\ge 0}$ be a $C_0$-semigroup on a Banach space $X$ and 
$x\in X$. The \emph{orbit} of $x$ under $(T_t)_{t\ge 0}$ is defined by $orb(x, (T_t)_{t\geq 0})=\{T_tx: t\ge 0\}$.
We say that $(T_t)_{t\ge 0}$ is \emph{hypercyclic} if there is some $x\in X$ whose orbit under $(T_t)_{t\ge 0}$ is dense in $X$. In such case, $x$ is called a \emph{hypercyclic vector} for $(T_t)_{t\ge 0}$.

The following result concerned with the equivalent characterizations of hypercyclicity for $C_0$-semigroups was proved in \cite{CMP2007}, see \cite{GP2011}*{Theorem 7.26} for the version stated here.

\begin{thm} 
Let $(T_t)_{t\geq 0}$ be a $C_0$-semigroup on a Banach space $X$ and $x\in X$. Then the following assertions are equivalent:
\begin{enumerate}
    \item $x$ is hypercyclic for $(T_t)_{t\geq 0}$;
    \item $x$ is hypercyclic for $T_s$ for some $s>0$;
    \item $x$ is hypercyclic for $T_s$ for every $s>0$.
\end{enumerate}
\end{thm}

Let $A\subset[0,\infty)$ be a Lebesgue measurable set.
The \emph{lower and upper densities} of $A$ are defined as 
\[
\lDens(A)=\liminf_{t\to\infty}\frac{\mu(A\cap[0,t])}{t}\text{ and }
\uDens(A)=\limsup_{t\to\infty}\frac{\mu(A\cap[0,t])}{t},
\]
where $\mu$ is the Lebesgue measure on $[0,\infty)$.
If $\lDens(A)=\uDens(A)$, then we say that the \emph{density} of $A$ exists, and denote it by $\Dens(A)$.

The following result on the relationship between ``$\mathrm{Dens}$'' and ``$\mathrm{dens}$'' is essential contained in \cite{ABMP2013}*{Lemma 2.4}, see also \cite{BBPW2018}*{Lemma 3}.

\begin{lem} \label{Dens-dens}
Let $(T_t)_{t\geq 0}$ be a $C_0$-semigroup on a Banach space $X$.
For each $s>0$, let $C_s= \sup_{t\in [0,s]} \Vert  T_t\Vert$.
Then for every $x\in X$ and $\varepsilon>0$, 
\begin{enumerate}
\item $\lDens(\{t\geq 0\colon \Vert T_tx\Vert<\varepsilon\})\leq \ldens(\{j\in\mathbb{N}\colon \Vert (T_s)^j x\Vert <C_s \varepsilon\})$;
\item $\ldens(\{j\in\mathbb{N}\colon \Vert (T_s)^jx\Vert < \varepsilon\})\leq \lDens(\{t\geq 0\colon \Vert T_tx\Vert<C_s\varepsilon\})$;
\item $\uDens(\{t\geq 0\colon \Vert T_tx\Vert<\varepsilon\})\leq \udens(\{j\in\mathbb{N}\colon \Vert (T_s)^jx\Vert  <C_s \varepsilon\})$;
\item $\udens(\{j\in\mathbb{N}\colon \Vert (T_s)^jx\Vert < \varepsilon\})\leq \uDens(\{t\geq 0\colon \Vert T_tx\Vert<C_s\varepsilon\})$.
\end{enumerate}
\end{lem}

Similar to that for a single operator, we can define mean-L-stability and mean-L-unstability for a $C_0$-semigroup as follows. Let $(T_t)_{t\ge 0}$ be a $C_0$-semigroup on a Banach space $X$.
We say that $X$ is \emph{mean-L-stable} for $(T_t)_{t\ge 0}$ if for every $\epsilon>0$ there exists a $\delta>0$ such that for any $x, y\in X$ with $d(x,y)<\delta$, one has
\[
    \uDens(\{t\ge 0: \Vert T_tx-T_ty\Vert \ge \epsilon \} )<\epsilon,
\]
and \emph{mean-L-unstable} for $(T_t)_{t\ge 0}$ if there exists a $\delta>0$ such that for any non-empty
open subset $U$ of $X$ there exist $x, y\in U$ such that
\[
    \uDens(\{t\ge 0: \Vert T_tx - T_ty\Vert >\delta \}) \ge \delta.
\]

The following proposition gives equivalent characterizations for mean-L-stability (resp. mean-L-unstability).
\begin{prop}\label{prop:C0meanL}
Let $(T_t)_{t\geq 0}$ be a $C_0$-semigroup on a Banach space $X$. Then the following assertions are equivalent:
\begin{enumerate}
    \item \label{(T_t)} the $C_0$-semigroup $(T_t)_{t\geq 0}$ is mean-L-stable (resp. mean-L-unstable);
    \item \label{some T_s} the operator $T_s$  is mean-L-stable (resp. mean-L-unstable) for some $s>0$;
    \item \label{every T_s} the operator $T_s$  is mean-L-stable (resp. mean-L-unstable) for every $s>0$.
\end{enumerate}
\end{prop}
\begin{proof}
$(\ref{every T_s})  \Rightarrow (\ref{some T_s})$: It is clear.

$(\ref{(T_t)}) \Rightarrow (\ref{every T_s})$: For every $s>0$, and let $C_s=\sup_{t\in [0,s]}\Vert T_t \Vert$. Since $(T_t)_{t\geq 0}$ is mean-L-stable, then for any $\epsilon>0$, there exists a $\delta>0$, such that for any $x,y\in X$ with $d(x,y)<\delta$, 
\[
    \uDens(\{t\ge 0: \Vert T_tx- T_ty\Vert \ge \tfrac{\epsilon}{C_s} \} )<\epsilon,
\]
then by Lemma \ref{Dens-dens}, one has
\[
\udens(\{j\in\mathbb{N}\colon \Vert (T_s)^jx-(T_s)^jy\Vert \ge \varepsilon\})\leq \uDens(\{t\geq 0\colon \Vert T_tx-T_ty\Vert\ge \tfrac{\epsilon}{C_s}\})<\epsilon.
\]
This shows that $T_s$ is mean-L-stable.

$(\ref{some T_s}) \Rightarrow (\ref{(T_t)}) $:
We may assume that $T_s$ is mean-L-stable for some $s>0$.
Put $C_s=\sup_{t\in [0,s]}\Vert T_t \Vert$.
Then for any $\epsilon>0$, there exists a $\delta>0$ such that for any $x,y\in X$ with $d(x,y)<\delta$, 
\[
    \udens(\{j\in \mathbb{N}: \Vert (T_s)^jx-(T_s)^jy \Vert \ge \tfrac{\epsilon}{C_s} \} )<\epsilon,
\]
thus by Lemma \ref{Dens-dens}, one has
\[
\uDens(\{t\ge 0: \Vert T_tx-T_ty \Vert \ge \epsilon \}) \le \udens(\{j\in \mathbb{N}: \Vert (T_s)^jx-(T_s)^jy \Vert \ge \tfrac{\epsilon}{C_s} \} )<\epsilon.
\]
This shows that $(T_t)_{t\ge 0}$ is mean-L-stable.

By a similar proof, Proposition \ref{prop:C0meanL} for mean-L-unstable case holds.
\end{proof}

In order to obtain a parallel result to Theorem \ref{thm:main-result}, a similar parameter needs to be introduced as follows. If $(T_t)_{t\geq 0}$ is a hypercyclic $C_0$-semigroup on a Banach space $X$,   
for each $R>0$, define 
\[
c_R((T_t)_{t\geq 0}) =\sup_{x\in HC((T_t)_{t\geq 0})} \uDens(\{t\in\mathbb{R}_{\geq0 }\colon \Vert T_tx\Vert<R\}),
\]
and  
\[
c((T_t)_{t\geq 0} )=\sup_{R>0} c_R((T_t)_{t\geq 0}).
\]

The following proposition plays a similar role in the proof of Theorem \ref{thm:C-semigroup-result} as that Proposition \ref{prop:c-T-eq} in the proof of Theorem \ref{thm:main-result}, and can be proved in a similar way as Proposition~\ref{prop:C0meanL}, hence we omit the proof.
\begin{prop} 
Let $(T_t)_{t\geq 0}$ be a hypercyclic $C_0$-semigroup on a Banach space $X$. 
Then $c((T_t)_{t\geq 0} )=c(T_s)$ for every $s>0$.
\end{prop}
We also need to introduce some concepts in terms of density for $C_0$-semigroups, as we did in Section 1 for a single operator.
Let $(T_t)_{t\geq 0}$ be a $C_0$-semigroup on a Banach space $X$.
A vector $x\in X$ is said to be \emph{distributionally irregular of type $1$} for $(T_t)_{t\geq 0}$ if there exist two measurable subsets $A$ and $B$ of $\mathbb{R}_{\geq 0}$ 
such that $\uDens(A)=\uDens(B)=1$, $\lim\limits_{A\ni t\to\infty} \Vert T_t x\Vert =0$ and $\lim\limits_{B\ni t\to\infty} \Vert T_t x\Vert =\infty$.  
A vector $x\in X$ is said to be \emph{distributionally irregular of type $2\frac{1}{2}$} for $(T_t)_{t\geq 0}$ 
if there exist two measurable  subsets $A$ and $B$ of $\mathbb{R}_{\geq 0}$ such that 
$\uDens(A)\uDens(B)>0$, $\lim\limits_{A\ni t\to\infty} \Vert T_t x\Vert =0$ and $\lim\limits_{B\ni t\to\infty} \Vert T_t x\Vert =\infty$. 
A vector $x\in X$ is said to be \emph{asymptotic to zero with density one} if there exists a measurable subset $A$ of $\mathbb{R}_{\geq 0}$ with $\Dens(A)=1$ such that $\lim\limits_{A\ni t\to\infty}\Vert T_t x\Vert =0$,
and \emph{divergent to infinity with density one}
if there exists a measurable subset $A$ of $\mathbb{R}_{\geq 0}$ with $\Dens(A)=1$ such that $\lim\limits_{A\ni t\to\infty}\Vert T_t x\Vert =\infty$.

Now we are ready to state our main result for the case of $C_0$-semigroups, whose proof is similar to that of Theorem \ref{thm:main-result}, and we omit it.
\begin{thm} \label{thm:C-semigroup-result}
If $(T_t)_{t\geq 0}$ is a hypercyclic $C_0$-semigroup on a Banach space $X$, then exactly one of the following four assertions holds: 
\begin{enumerate}
    \item every vector in $X$ is asymptotic to zero with density one;
    \item generic vectors in $X$ are distributionally irregular of type $1$;
    \item generic vectors in $X$ are distributionally irregular of type $2\frac{1}{2}$, and no hypercyclic vector is distributionally irregular of type $1$;
    \item every hypercyclic vector in $X$ is divergent to infinity with density one.
\end{enumerate}
\end{thm}

There are some examples in \cite{ABMP2013} to show that 
the case (2) in Theorem \ref{thm:C-semigroup-result} can occur. In the following, we will give more examples to show that the other three cases in Theorem \ref{thm:C-semigroup-result} can occur.

\subsection{Unilateral translation semigroup}
By an \emph{admissible weight function} on $\mathbb{R}_{\geq 0}$, we mean a measurable function 
$\rho\colon \mathbb{R}_{\geq 0}\to \mathbb{R}$ satisfying the following two conditions:
\begin{enumerate}
\item $\rho(t)>0$ for all $t\in\mathbb{R}_{\geq 0}$;
\item there exist constants $M\geq 1$ and $\omega>0$ such that $\rho(s)\leq Me^{\omega t} \rho(t+s)$ for all $s\in\mathbb{R}_{\geq t}$ and all $t>0$.
\end{enumerate}

Let $1\leq p<\infty$ and $\rho$ be an admissible weight function on $\mathbb{R}_{\geq 0}$. 
We consider the separable Banach space of $p$-integrable functions in the Lebesgue sense as 
\[
L_\rho^p(\mathbb{R}_{\geq 0}):=\{ f \colon \mathbb{R}_{\geq 0}\to \mathbb{R}\text{ measurable}, \Vert f\Vert_p<\infty\},
\]
where 
\[
\Vert f\Vert_p = \biggl(\int_{\mathbb{R}_{\geq 0}} |f(s)|^p \rho(s)\dd s\biggr)^{1/p}.
\]
The \emph{translation semigroup} on $L_\rho^p(\mathbb{R}_{\geq 0})$ defined by $(T_t f)(x)=f(x+t)$, $t,x\in\mathbb{R}_{\geq 0}$, 
is a well-defined $C_0$-semigroup by the definition of admissible weight.

Let $v=(v_j)_{j\in\mathbb{N}}$ be a positive weight sequence with $\sup_{j\in\mathbb{N}}\frac{v_j}{v_{j+1}}<\infty$.
Following \cite{BP2012}, we construct an admissible weight function induced by $v$. 
For each $t\in (n-1,n]$ for some $n\in\mathbb{N}$, 
let $\rho_v(t)=v_n$, and $\rho_v(0)=v_1$.
Then it is easy to see that $\rho_v$ is an admissible weight function.

For each $f\in L_\rho^p (\mathbb{R}_{\geq 0})$ and $n\in\mathbb{N}$,
define $x_n^f=\Bigl(\int_{[n-1,n]} |f(t)|^p \dd t\Bigr)^{1/p}$.
Let $x^f=(x_n^f)_{n\in\mathbb{N}}$.
Then  
\[
\Vert x^f\Vert_p^p = \sum_{n=1}^\infty |x^f_n|^p v_n 
= \sum_{n=1}^\infty \int_{[n-1,n]} |f(t)|^p \rho(t) \dd t
=\int_{\mathbb{R}_{\geq 0}} |f(t)|^p \rho(t) \dd t 
=\Vert f\Vert_p^p.
\]

Combining Examples 4.4, 4.7 and 7.10 in \cite{GP2011}, 
we have the following result.

\begin{prop}
Let $v=(v_j)_{j\in\mathbb{N}}$ be a positive weight sequence with $\sup_{j\in\mathbb{N}}\frac{v_j}{v_{j+1}}<\infty$.
If $(\ell^p(v),B)$ is hypercyclic (resp. mixing), then $(L_{\rho_v}^{p} (\mathbb{R}_{\geq 0}), (T_t)_{t\geq 0})$ is hypercyclic (resp. mixing).
\end{prop}

We now have the following result, which, when combined with Example~\ref{exam:case-1} from the previous section, shows that
case (1) in Theorem~\ref{thm:C-semigroup-result} can occur.

\begin{prop}
Let $v=(v_j)_{j\in\mathbb{N}}$ be a positive weight sequence with $\sup_{j\in\mathbb{N}}\frac{v_j}{v_{j+1}}<\infty$.
If $(\ell^p(v,\mathbb{N}),B)$ is mixing and every vector is asymptotic to zero with density one, 
then $(L_{\rho_v}^p (\mathbb{R}_{\geq 0}), (T_t)_{t\geq 0})$ 
is mixing and every vector is asymptotic to zero with density one.
\end{prop} 

\subsection{Bilateral translation semigroup}
Similar to that in previous subsection, by an \emph{admissible weight function} on $\mathbb{R}$, we mean a measurable function 
$\rho\colon \mathbb{R}\to \mathbb{R}$ satisfying the following two conditions:
\begin{enumerate}
\item $\rho(t)>0$ for all $t\in\mathbb{R}$;
\item there exist constants $M\geq 1$ and $\omega>0$ such that $\rho(s)\leq Me^{\omega t} \rho(t+s)$ for all $s\in\mathbb{R}$ and all $t>0$.
\end{enumerate}

Let $1\leq p<\infty$ and  $\rho$ be an admissible weight function on $\mathbb{R}$. 
Just as that in Subsection 4.1, we consider the separable Banach space of $p$-integrable functions in the Lebesgue sense as 
\[
L_\rho^p(\mathbb{R}):=\{ f \colon \mathbb{R}\to \mathbb{R}\text{ measurable}, \Vert f\Vert_p<\infty\},
\]
where 
\[
\Vert f\Vert_p = \biggl(\int_{\mathbb{R}} |f(s)|^p \rho(s)\dd s\biggr)^{1/p}.
\]
Similarly, the \emph{translation semigroup} on $L_\rho^p(\mathbb{R})$ defined by $(T_t f)(x)=f(x+t)$, $t,x\in\mathbb{R}$, 
is a well-defined $C_0$-semigroup by the definition of admissible weight.

Let $v=(v_j)_{j\in\mathbb{Z}}$ be a positive weight sequence with $\sup_{j\in\mathbb{Z}}\frac{v_j}{v_{j+1}}<\infty$.
For each $t\in (n-1,n]$ for some $n\in\mathbb{Z}$, let $\rho_v(t)=v_n$.
Clearly $\rho_v$ is an admissible weight function.

For each $f\in L_\rho^p (\mathbb{R})$ and $n\in\mathbb{Z}$,
define $x_n^f=\Bigl(\int_{[n-1,n]} |f(t)|^p \dd t\Bigr)^{1/p}$.
Let $x^f=(x_n^f)_{n\in\mathbb{Z}}$.
Similar to the discussion in the previous subsection, we know that 
$\Vert x^f\Vert_p^p =\Vert f\Vert_p^p$. 
Then similar to the above construction, we have the following:

\begin{prop}
Let $v=(v_j)_{j\in\mathbb{Z}}$ be a positive weight sequence with $\sup_{j\in\mathbb{Z}}\frac{v_j}{v_{j+1}}<\infty$.
If $(\ell^p(v,\mathbb{Z}),B)$ is hypercyclic and every non-zero vector is distributionally irregular of type 1,
then $(L_{\rho_v}^p (\mathbb{R}), (T_t)_{t\geq 0})$ 
is hypercyclic and  every non-zero vector is distributionally irregular of type 1.
\end{prop} 

\begin{prop}
Let $v=(v_j)_{j\in\mathbb{Z}}$ be a positive weight sequence with $\sup_{j\in\mathbb{Z}}\frac{v_j}{v_{j+1}}<\infty$.
If $(\ell^p(v,\mathbb{Z}),B)$ is hypercyclic and  
for every non-zero vector $x\in \ell^p(v,\mathbb{Z})$
there exists $\tau_x>0$ such that 
\[
    \ldens(\{n\in\mathbb{N}\colon \Vert B^n x\Vert\geq \tau_x\})\geq \tfrac{3}{7},
\]
and for any  vector with only finite non-zero coordinates $x\in \ell^p(v,\mathbb{Z})$,
\[
   \udens(\{n\in\mathbb{N}\colon \Vert B^n x\Vert\leq \varepsilon\})\geq \tfrac{4}{7}, 
\]
then  $(L_{\rho_v}^p (\mathbb{R}), (T_t)_{t\geq 0})$ 
is hypercyclic and for every non-zero function $f\in L_{\rho_v}^p (\mathbb{R})$
there exists $\tau_f>0$ such that 
\[
    \lDens(\{t\geq 0\colon \Vert T_t f\Vert\geq \tau_f\})\geq \tfrac{3}{7},
\]
for any function $f$ with only bounded support and $\varepsilon>0$, 
\[
   \uDens(\{t\geq 0 \colon \Vert T_t f\Vert\leq \varepsilon\})\geq \tfrac{4}{7}.
\]
\end{prop}

By combining the above result with Example~\ref{exam:case3} from the previous section, we see that case (3) in Theorem \ref{thm:C-semigroup-result} can occur.

\begin{prop}
Let $v=(v_j)_{j\in\mathbb{Z}}$ be a positive weight sequence with $\sup_{j\in\mathbb{Z}}\frac{v_j}{v_{j+1}}<\infty$.
If $(\ell^p(v,\mathbb{Z}),B)$ is hypercyclic and every non-zero vector is divergent to infinity with density one, 
then $(L_{\rho_v}^p (\mathbb{R}), (T_t)_{t\geq 0})$ 
is hypercyclic and every non-zero vector is divergent to infinity with density one.
\end{prop} 

By combining the above result with Example~\ref{exam:case4} from the previous section, we see that case (4) in Theorem \ref{thm:C-semigroup-result} can occur.

\bigskip 

\noindent \textbf{Acknowledgment.}
J. Li was supported in part by NSF of China (12222110). 
X. Wang was partially supported by STU Scientific Research Initiation Grant (SRIG, No. NTF22020) and NSF of China (12301230). 
J. Zhao was supported by NSF of China (12301226).
The authors would like to thank the referee for the careful reading and helpful suggestions.

The authors would like to thank Mohamed Amine Aouichaoui (University of Monastir, Tunisia) for pointing out a minor error in the definition of the weights $(v_j)$ in Example 3.7 in the published version of the paper, as well as several other minor typographical errors. 
These minor oversights have been corrected in the latest version of the paper available on arXiv.

\end{document}